\documentclass{amsart}
\usepackage[section]{placeins}

\title{Reflective and quasi-reflective Bianchi groups}

\author{Mikhail Belolipetsky}\thanks{Belolipetsky partially supported by a CNPq research grant.}
\address{
IMPA\\
Estrada Dona Castorina 110\\
22460-320 Rio de Janeiro, Brazil}
\email{mbel@impa.br}

\author{John Mcleod}
\address{
Department of Mathematical Sciences\\
Durham University\\
South Rd\\
Durham DH1 3LE, United Kingdom
}
\email{j.a.mcleod@durham.ac.uk}

\date{\today}

\usepackage{amssymb}
\usepackage{amsmath}
\usepackage{amsthm}
\usepackage{graphicx}
\usepackage{fullpage}
\usepackage{verbatim}
\usepackage{tikz}
\usepackage{float}

\newtheorem{theorem}{Theorem}[section]
\newtheorem{proposition}[theorem]{Proposition}

\newtheorem{lemma}[theorem]{Lemma}
\theoremstyle{definition}
\newtheorem{definition}[theorem]{Definition}

\newtheorem{remark}[theorem]{Remark}

\newcommand\Am{O_m}
\newcommand\Z{{\mathbb{Z}}}
\newcommand\Q{{\mathbb{Q}}}
\newcommand\C{{\mathbb{C}}}
\newcommand\Hy{{\mathbb{H}}}

\newcommand\isom{\operatorname{Isom}(\mathbb{H}^3)}
\newcommand\dist{\operatorname{dist}}
\newcommand\sym{\operatorname{Sym}(P)}
\newcommand\PGL{\operatorname{PGL_2}}
\newcommand\PSL{\operatorname{PSL_2}}

\newcommand\Bi{Bi(m)}
\newcommand\hBi{\widehat{Bi}(m)}

\begin{document}
\begin{abstract}
A discrete subgroup of the group of isometries of the hyperbolic space is called reflective if up to a finite index it is generated by reflections in hyperplanes. The main result of this paper is a complete classification of the reflective (and quasi-reflective) subgroups among the Bianchi groups and their extensions.
\end{abstract}
\maketitle

\section{Introduction}
The study of the Bianchi groups has a distinguished heritage, being a contemporary application of the work of Klein and Fricke (as part of the Erlangen program) on elliptic modular functions. Bianchi's early work was concerned with differential geometry and functional theory, but by 1890 he was interested in M\"obius transformations over integral values of imaginary quadratic fields, possibly influenced by Klein's solution to the quintic equation \cite{klein1888}. Initially applying geometric methods to number theoretic problems about these transformations Bianchi \cite{bianchi1891} then moved towards the more geometric question of considering subgroups of these groups that are generated by reflections in hyperplanes. This research lead to his famous paper \cite{bianchi1892} wherein he proves that for $m \le 19$ ($m \neq 14, 17$) the Bianchi groups $Bi(m)$ are \emph{reflective} (where $m$ is a square-free positive integer).

These results are widely known, but new examples were not forthcoming until 1987 when three papers appeared in the proceedings \cite{yaroslavl1988}. Here we see the study of reflective Bianchi groups drawn under the wider program of classification of reflective hyperbolic lattices initiated by {\`E}. B. Vinberg. He uses extensions of the Bianchi groups whose automorphism groups are contained in automorphism groups of particular quadratic forms, and proves that whether one is reflective depends on the order of the elements in the ideal class group of the underlying number field \cite{vinberg90}.

Within this framework, Shaiheev \cite{shaiheev90} uses an algorithm of Vinberg \cite{vinberg72(2)} to prove that the Bianchi groups are reflective for $m \le 21$, $m \neq 14, 17$, draws the fundamental domains of corresponding reflection subgroups, and uses Vinberg's result about the ideal class group \cite{vinberg90} to confirm that the total number of reflective Bianchi groups (and extended Bianchi groups) is finite. He also proves that these extensions of the Bianchi groups are reflective for $m \le 30$, $m \neq 22, 23, 26, 29$. (Our computation shows that there is a small mistake in these results and $m=21$ should be excluded from the first list.)

In the final paper, Shvartsman \cite{shvartsman90} applies a nice geometric argument in the model of hyperbolic 3-space expounded in the preceding work to prove that the Bianchi groups are in fact non-reflective for $m \ge 22$, $m \equiv 1,2\text{ (mod 4)}$. He extended his result in 1998 by showing that when a Bianchi group is reflective the class number of the underlying number field is bounded above by 4 \cite{shvartsman98}.

After this work only $43$ number fields remained as candidates for new examples. The first purpose of our paper is to close this gap and thus complete the classification of the reflective Bianchi groups.

A wider classification continued into the 1990s, with a paper of Ruzmanov \cite{ruzmanov90} introducing the \emph{quasi-reflective} Bianchi groups, which are also known as \emph{parabolic reflection groups} (cf.~\cite{nikulin96}). A quasi-reflective group $\Gamma_{QR}$ can be viewed as an infinite index extension of a reflection group, where the fundamental polyhedron of the reflection group has infinite volume and the action of the (infinite) symmetry group of the polyhedron preserves a particular horosphere on which it acts by affine transformations. Ruzmanov showed that for $m \le 51$, or $m \equiv 1,2\text{ (mod 4)}$, the Bianchi group $Bi(m)$ is quasi-reflective for  $m = 14, 17, 23, 31, 39$. In Nikulin's paper \cite{nikulin96} it is proved that there are only finitely many quasi-reflective lattices in any dimension. Arguably the most interesting example of a quasi-reflective group appears in dimension $25$, where the corresponding subgroup of affine transformations is the group of automorphisms of the Leech lattice. This example was first discovered by Conway \cite{conway83}.

In this paper we complete the classification of quasi-reflective Bianchi groups. This is accomplished by exploiting the high degree of similarity between the reflective and quasi-reflective cases. We extend the results of the first author on finiteness of reflective arithmetic Kleinian groups to the quasi-reflective setting and then test each of the remaining cases with a combination of an algorithm of Vinberg and simple geometric criteria.

Bianchi groups constitute an important subclass of non-cocompact arithmetic Kleinian groups. The questions about reflectivity were also studied in the cocompact case but here the progress was much slower. The breakthrough occurred in recent years when Long-Maclachlan-Reid~\cite{longmaclachlanreid06} in dimension $2$ and Agol~\cite{agol06} in dimension $3$ showed that there are only finitely many underlying number fields defining arithmetic reflection groups in these dimensions. The first author made the method of \cite{agol06} effective by proving the best known bound on the degree of the fields of definition of such groups for $n = 3$ \cite{belolipetsky09}. His method also provides a list of possible fields of definition of reflective Bianchi groups, but the result is less precise than Shvartsman's.


Various questions related to reflective Bianchi groups were studied by many authors. We would like to mention the papers by Erlstrodt-Grunewald-Mennicke \cite{EGM}, Scharlau-Walhorn \cite{scharlauwalhorn92}, Vulakh \cite{vulakh93} and James-Maclachlan \cite{jamesmaclachlan96}. We will refer to some results from these sources in our work.

The novelty of the present paper is in the use of combination of the spectral methods of \cite{agol06} and \cite{belolipetsky09} with more classical arguments mentioned above. This allows us to consider all cases in a uniform fashion and to close the gaps in the classification. We also provide some  proofs which were missing in the literature and correct several mistakes which appeared in published papers. As a byproduct we found a new example of a reflective extended Bianchi group which does not appear in \cite{shaiheev90}. It correspond to $m = 33$, the fundamental polyhedron of the reflection subgroup has 24 vertices (including 2 cusps), 37 edges and 15 faces. Its Coxeter diagram is presented in Figure~\ref{m=33}. This example is a part of the main results of this paper which are stated in Theorems~\ref{thm1}--\ref{thm3}.

The paper is organised as follows. In Section~\ref{sec:statements} we give precise statements of the main results. The background material on the hyperbolic $3$-space, Bianchi groups, reflections and Vinberg's algorithm is collected in Section~\ref{background}. In Section~\ref{reflbackground} we discuss the algebraic structure of reflective and quasi-reflective groups, and the relation between reflectivity of a Bianchi group and arithmetic of the field $\Q[\sqrt{-m}]$. Here we also extend Vinberg's result about the class groups of the fields of definition of reflective Bianchi groups to the quasi-reflective case. In Section~\ref{QRbackground} we generalise the spectral method of \cite{agol06} to prove an effective finiteness result for the quasi-reflective Bianchi groups. Analogous finiteness theorem for the reflective groups was proved in \cite{belolipetsky09}. Section~\ref{sec:lemmas} contains several arithmetic and geometric lemmas about the cusps of the reflection subgroups of (extended) Bianchi groups. These lemmas establish properties of the set of cusps which are used in the next section to demonstrate non-reflectivity and non-quasi-reflectivity of certain groups. Combined with Vinberg's algorithm and one other known criterion of non-reflectivity this allows us to finish the proof of the main theorems in Section~\ref{(Q)reflbianchi}. Details of some computations from this section and further corollaries of the main theorems will appear in the second author's PhD thesis \cite{mcleod_thesis}.

\medskip

\noindent\emph{Acknowledgements.} We would like to thank the referees of this paper for their helpful comments and corrections. Part of this work was done while the authors were visiting MSRI at Berekeley, we would like to thank MSRI for the hospitality and support.

\section{The statement of the main results}\label{sec:statements}

Let $\Am$ be the ring of integers of the imaginary quadratic field $K_m = \Q[\sqrt{-m}]$ (where $m$ is a square-free positive integer). Following Vinberg \cite{vinberg90} we define the \emph{Bianchi group} $Bi(m)$ by
\begin{equation*}
Bi(m) = \mathrm{PGL}_2(\Am) \rtimes \langle\tau\rangle,
\end{equation*}
where $\tau$ is an element of order 2 that acts on $\mathrm{PGL}_2(\Am)$ as complex conjugation.

The group $Bi(m)$ can be regarded in a natural way as a discrete group of isometries of the hyperbolic $3$-space $\Hy^3$ (see below). Together with $Bi(m)$ we will also consider the \emph{extended Bianchi group} $\widehat{Bi}(m)$, which is the maximal discrete subgroup of $\isom$ containing $\mathrm{PGL}_2(\Am)$ (cf. \cite{allan66}). The group $\widehat{Bi}(m)$ is defined by
\begin{equation*}
\widehat{Bi}(m) = \widehat{\mathrm{PGL}}_2(\Am) \rtimes \langle\tau\rangle,
\end{equation*}
where $\widehat{\mathrm{GL}}_2(\Am)$ denotes the group of matrices in $\mathrm{GL}_2(K_m)$ which, under the natural action in the space $K_m^2$, multiply the lattice $\Am^2$ by the fractional ideal of the ring $\Am$ (whose square is automatically a principle ideal).

We refer to Section~\ref{reflbackground} for definitions and basic properties of the reflective and quasi-reflective subgroups.
In this paper we give a full list of the reflective and quasi-reflective groups $Bi(m)$ and $\widehat{Bi}(m)$, and provide an explicit description of corresponding reflection subgroups in terms of Coxeter diagrams (see Section~\ref{background} for a definition of a Coxeter diagram). Our main results are given in the following three theorems.

\begin{theorem}\label{thm1}
The Bianchi groups $Bi(m)$ are reflective for $m \le 19$, $m \neq 14, 17$, and this list is complete.
\end{theorem}

\begin{theorem}\label{thm2}
The extended Bianchi groups $\widehat{Bi}(m)$ are reflective for $m \le 21$, $m = 30$, $33$ and $39$, and this list is complete.
\end{theorem}

\begin{theorem}\label{thm3}
The Bianchi groups $Bi(m)$ are quasi-reflective for $m=14$, $17$, $23$, $31$ and $39$, and this list is complete. The only quasi-reflective extended Bianchi groups are $\widehat{Bi}(23)$ and $\widehat{Bi}(31)$.
\end{theorem}

The Coxeter diagrams of the groups from Theorems~\ref{thm1} and \ref{thm2} are presented in Figures~\ref{fundamentaldomains}-\ref{m=33}. The diagrams in Figure~\ref{fundamentaldomains} correspond to the groups which previously appeared in the papers by Shaiheev and Ruzmanov (\cite{shaiheev90}, \cite{ruzmanov90}), while the example in Figure~\ref{m=33} is new. We note that there are some mistakes in the data in \cite{shaiheev90} corresponding to $m = 17$ and $21$. In particular, in contrary to the previous claims we find that the group $Bi(21)$ is not reflective. We correct these mistakes in Figure~\ref{fig1} and Table~\ref{fundamentalvectors} in Section~\ref{(Q)reflbianchi}.
\begin{figure}[!ht]
\centering
\input{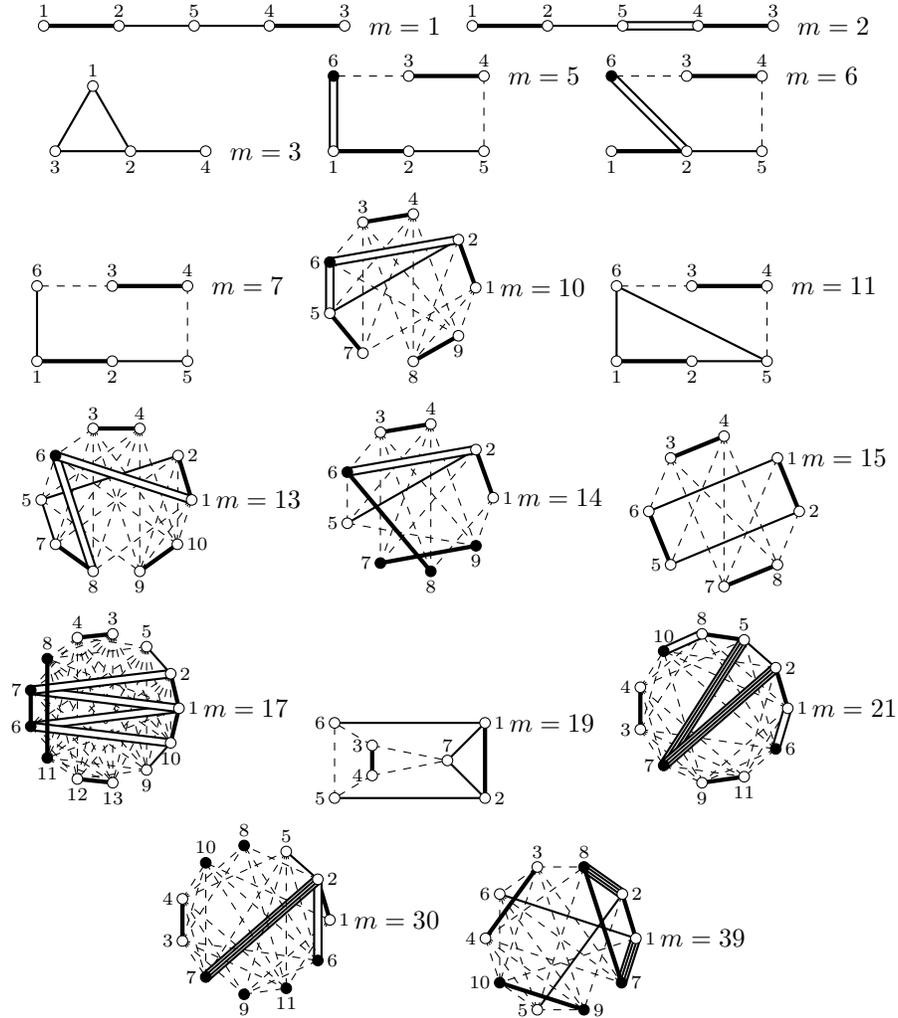}
\caption{\label{fig1} Coxeter diagrams of the fundamental domains of the reflective extended Bianchi groups $\widehat{Bi}(m)$ studied by Shaiheev and Ruzmanov. Vertices that are filled represent reflections which are in the group $\widehat{Bi}(m)$ but not in $Bi(m)$. \label{fundamentaldomains}}
\end{figure}

\begin{figure}[!ht]
\centering
  \begin{tikzpicture}
    \coordinate (one) at (-0.81,1.83);
    \coordinate (two) at (-1.49,1.34);
    \coordinate (three) at (-1.9,0.62);
    \coordinate (four) at (-1.99,-0.21);
    \coordinate (five) at (-1.73,-1);
    \coordinate (ten) at (-1.18,-1.62);
    \coordinate (eight) at (-0.42,-1.96);
    \coordinate (seven) at (0.42,-1.96);
    \coordinate (thirteen) at (1.18,-1.62);
    \coordinate (six) at (1.73,	-1);
    \coordinate (twelve) at (1.99,-0.21);
    \coordinate (fourteen) at (1.9,0.62);
    \coordinate (eleven) at (1.49,1.34);
    \coordinate (fifteen) at (0.81,1.83);
    \coordinate (nine) at (0,2);

    \draw[dashed] (one) -- (seven);
    \draw[dashed] (one) -- (thirteen);
    \draw[dashed] (two) -- (ten);
    \draw[dashed] (two) -- (eleven);
    \draw[dashed] (two) -- (twelve);
    \draw[dashed] (two) -- (thirteen);
    \draw[dashed] (two) -- (fourteen);
    \draw[dashed] (two) -- (fifteen);
    \draw[dashed] (three) -- (six);
    \draw[dashed] (three) -- (seven);
    \draw[dashed] (three) -- (eight);
    \draw[dashed] (three) -- (nine);
    \draw[dashed] (three) -- (ten);
    \draw[dashed] (three) -- (eleven);
    \draw[dashed] (three) -- (twelve);
    \draw[dashed] (three) -- (thirteen);
    \draw[dashed] (three) -- (fourteen);
    \draw[dashed] (three) -- (fifteen);
    \draw[dashed] (four) -- (five);
    \draw[dashed] (four) -- (seven);
    \draw[dashed] (four) -- (eight);
    \draw[dashed] (four) -- (nine);
    \draw[dashed] (four) -- (ten);
    \draw[dashed] (four) -- (eleven);
    \draw[dashed] (four) -- (twelve);
    \draw[dashed] (four) -- (thirteen);
    \draw[dashed] (four) -- (fourteen);
    \draw[dashed] (four) -- (fifteen);
    \draw[dashed] (five) -- (six);
    \draw[dashed] (five) -- (eight);
    \draw[dashed] (five) -- (nine);
    \draw[dashed] (five) -- (eleven);
    \draw[dashed] (five) -- (twelve);
    \draw[dashed] (five) -- (thirteen);
    \draw[dashed] (five) -- (fourteen);
    \draw[dashed] (five) -- (fifteen);
    \draw[dashed] (six) -- (seven);
    \draw[dashed] (six) -- (eight);
    \draw[dashed] (six) -- (ten);
    \draw[dashed] (six) -- (eleven);
    \draw[dashed] (six) -- (fourteen);
    \draw[dashed] (six) -- (fifteen);
    \draw[dashed] (seven) -- (nine);
    \draw[dashed] (seven) -- (eleven);
    \draw[dashed] (seven) -- (twelve);
    \draw[dashed] (seven) -- (thirteen);
    \draw[dashed] (seven) -- (fourteen);
    \draw[dashed] (seven) -- (fifteen);
    \draw[dashed] (eight) -- (twelve);
    \draw[dashed] (eight) -- (thirteen);
    \draw[dashed] (eight) -- (fourteen);
    \draw[dashed] (eight) -- (fifteen);
    \draw[dashed] (nine) -- (ten);
    \draw[dashed] (nine) -- (twelve);
    \draw[dashed] (ten) -- (eleven);
    \draw[dashed] (ten) -- (twelve);
    \draw[dashed] (ten) -- (thirteen);
    \draw[dashed] (ten) -- (fourteen);
    \draw[dashed] (ten) -- (fifteen);
    \draw[dashed] (eleven) -- (twelve);
    \draw[dashed] (eleven) -- (thirteen);
    \draw[dashed] (eleven) -- (fourteen);
    \draw[dashed] (twelve) -- (fifteen);
    \draw[dashed] (thirteen) -- (fifteen);

    \draw[ultra thick] (one) -- (two);
    \draw[thick, double distance = 2pt] (one) -- (six);
    \draw[thick, double distance = 2pt] (one) -- (eight);
    \draw[ultra thick] (one) -- (nine);
    \draw[thick] (two) -- (five);
    \draw[thick, double distance = 2pt] (two) -- (eight);
    \draw[thick] (two) -- (nine);
    \draw[ultra thick] (three) -- (four);
    \draw[thick, double distance = 2pt] (six) -- (nine);
    \draw[thick] (nine) -- (fourteen);
    \draw[ultra thick] (eleven) -- (fifteen);

    \filldraw[fill=white] (one) circle (2pt) node[above left] {\scriptsize 1};
    \filldraw[fill=white] (two) circle (2pt) node[above left] {\scriptsize 2};
    \filldraw[fill=white] (three) circle (2pt) node[left] {\scriptsize 3};
    \filldraw[fill=white] (four) circle (2pt) node[left] {\scriptsize 4};
    \filldraw[fill=white] (five) circle (2pt) node[below left] {\scriptsize 5};
    \filldraw (six) circle (2pt) node[below right] {\scriptsize 6};
    \filldraw (seven) circle (2pt) node[below] {\scriptsize 7};
    \filldraw (eight) circle (2pt) node[below] {\scriptsize 8};
    \filldraw[fill=white] (nine) circle (2pt) node[above] {\scriptsize 9};
    \filldraw (ten) circle (2pt) node[below left] {\scriptsize 10};
    \filldraw[fill=white] (eleven) circle (2pt) node[above right] {\scriptsize 11};
    \filldraw (twelve) circle (2pt) node[right] {\scriptsize 12};
    \filldraw (thirteen) circle (2pt) node[below right] {\scriptsize 13};
    \filldraw[fill=white] (fourteen) circle (2pt) node[right] {\scriptsize 14};
    \filldraw[fill=white] (fifteen) circle (2pt) node[above right] {\scriptsize 15};

  \end{tikzpicture}
\caption{\label{fig2} Coxeter diagram of the fundamental domain of the reflection subgroup of $\widehat{Bi}(33)$. The filled vertices represent reflections in $\widehat{Bi}(33)$ but not in $Bi(33)$. \label{m=33}}
\end{figure}
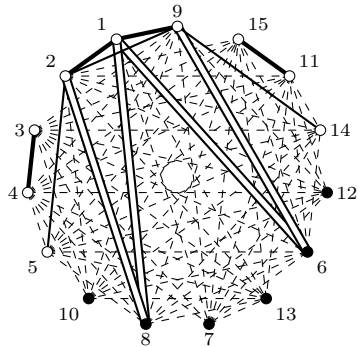

In Figure~\ref{Bi(23,31)} we present the parts of the infinite Coxeter diagrams of reflection subgroups in the groups $\widehat{Bi}(23)$ and $\widehat{Bi}(31)$ from Theorem~\ref{thm3}. A complete diagram of a corresponding reflection polyhedra can be obtained by tessellating the plane by the shaded parts of these diagrams. (The corresponding translations can be understood as the symmetries of the fundamental polyhedron in $\Hy^3$ which preserve a horosphere and act on it by affine translations.) We note that for the remaining $m=14$, $17$ and $39$ the extended Bianchi groups are reflective and hence the diagrams for the reflective subgroups of the groups $Bi(m)$ can be obtained from the corresponding diagrams for $\widehat{Bi}(m)$ by developing them around the cusp whose reflections are not in $Bi(m)$.

\begin{figure}[!ht]
\centering
\resizebox{36em}{!}{
  \input{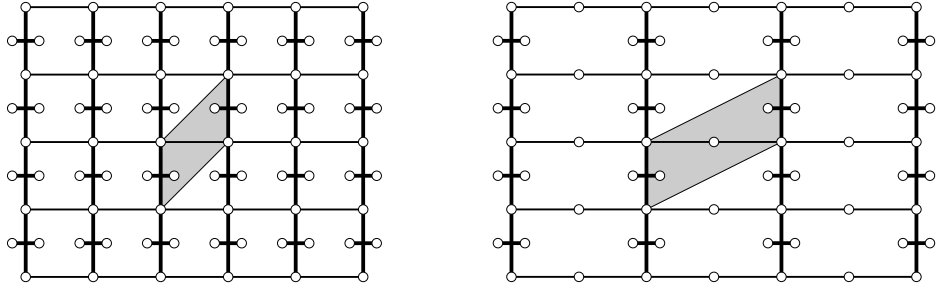}
  \hskip4em
  \begin{tikzpicture}
    \coordinate (one) at (0, 0);
    \coordinate (two) at (0, 1);
    \coordinate (three) at (-7, 3);
    \coordinate (four) at (-7, 2);
    \coordinate (five) at (1, 1);
    \coordinate (six) at (-1, 0);
    \coordinate (seven) at (2, 2);
    \coordinate (eight) at (-2, -1);
    \coordinate (nine) at (2, 1);
    \coordinate (ten) at (-2,0);
    \coordinate (eleven) at (4, 3);
    \coordinate (twelve) at (-4, -2);
    \coordinate (thirteen) at (1, 2);
    \coordinate (fourteen) at (-1, -1);
    \coordinate (fifteen) at (0, -1);
    \coordinate (sixteen) at (0, 2);    
    \coordinate (seventeen) at (-7, 1);    
    \coordinate (eighteen) at (3, 2);
    \coordinate (nineteen) at (2, 3);
    \coordinate (twenty) at (-2, -2);
    \coordinate (twentyone) at (-3, -1);
    \coordinate (twentytwo) at (4, 2);
    \coordinate (twentythree) at (-4, -1);
    \coordinate (twentyfour) at (1, 0);
    \coordinate (twentyfive) at (-1, 1);
    \coordinate (twentysix) at (2, 0);    
    \coordinate (twentyseven) at (-2, 1);    
    \coordinate (twentyeight) at (4, 4);    
    \coordinate (twentynine) at (-4, -3);    
    \coordinate (thirty) at (3, 3);    
    \coordinate (thirtyone) at (-3, -2);    
    \coordinate (thirtytwo) at (6, 4);    
    \coordinate (thirtythree) at (-6, -3);    
    \coordinate (thirtyfour) at (4, 1);    
    \coordinate (thirtyfive) at (-4, 0);    
    \coordinate (thirtysix) at (3, 1);    
    \coordinate (thirtyseven) at (6, 3);    
    \coordinate (thirtyeight) at (-6, -2);    
    \coordinate (thirtynine) at (-3, 0);    
    \coordinate (forty) at (6, 5);    
    \coordinate (fortyone) at (-6, -4);    
    \coordinate (fortytwo) at (2, 4);    
    \coordinate (fortythree) at (-2, -3);    
    \coordinate (fortyfour) at (0, -2);    
    \coordinate (fortyfive) at (0, 3);
    \coordinate (fortysix) at (5, 3);    
    \coordinate (fortyseven) at (5, 4);    
    \coordinate (e1) at (1, 3);    
    \coordinate (e2) at (-1, 2);    
    \coordinate (e3) at (-1, 3);    
    \coordinate (e4) at (-2, 2);    
    \coordinate (e5) at (-2, 3);    
    \coordinate (e6) at (3, 0);    
    \coordinate (e7) at (4, 0);    
    \coordinate (e8) at (1, -1);    
    \coordinate (e9) at (2, -1);    
    \coordinate (e10) at (3, -1);    
    \coordinate (e11) at (4, -1);    
    \coordinate (a1) at (-2.2, 1.5);    
    \coordinate (a2) at (-1.8, 1.5);    
    \coordinate (a3) at (-1.2, 1.5);    
    \coordinate (a4) at (-0.8, 1.5);    
    \coordinate (a5) at (-0.2, 1.5);    
    \coordinate (a6) at (0.2, 1.5);    
    \coordinate (a7) at (0.8, 1.5);    
    \coordinate (a8) at (1.2, 1.5);    
    \coordinate (a9) at (1.8, 1.5);    
    \coordinate (a10) at (2.2, 1.5);    
    \coordinate (a11) at (3.8, 1.5);    
    \coordinate (a12) at (4.2, 1.5);    
    \coordinate (b1) at (-2.2, 0.5);    
    \coordinate (b2) at (-1.8, 0.5);    
    \coordinate (b3) at (-1.2, 0.5);    
    \coordinate (b4) at (-0.8, 0.5);    
    \coordinate (b5) at (-0.2, 0.5);    
    \coordinate (b6) at (0.2, 0.5);    
    \coordinate (b7) at (0.8, 0.5);    
    \coordinate (b8) at (1.2, 0.5);    
    \coordinate (b9) at (1.8, 0.5);    
    \coordinate (b10) at (2.2, 0.5);    
    \coordinate (b11) at (3.8, 0.5);    
    \coordinate (b12) at (4.2, 0.5);    
    \coordinate (c1) at (-2.2, -0.5);    
    \coordinate (c2) at (-1.8, -0.5);    
    \coordinate (c3) at (-1.2, -0.5);    
    \coordinate (c4) at (-0.8, -0.5);    
    \coordinate (c5) at (-0.2, -0.5);    
    \coordinate (c6) at (0.2, -0.5);    
    \coordinate (c7) at (0.8, -0.5);    
    \coordinate (c8) at (1.2, -0.5);    
    \coordinate (c9) at (1.8, -0.5);    
    \coordinate (c10) at (2.2, -0.5);    
    \coordinate (c11) at (3.8, -0.5);    
    \coordinate (c12) at (4.2, -0.5);    
    \coordinate (d1) at (-2.2, 2.5);    
    \coordinate (d2) at (-1.8, 2.5);    
    \coordinate (d3) at (-1.2, 2.5);    
    \coordinate (d4) at (-0.8, 2.5);    
    \coordinate (d5) at (-0.2, 2.5);    
    \coordinate (d6) at (0.2, 2.5);    
    \coordinate (d7) at (0.8, 2.5);    
    \coordinate (d8) at (1.2, 2.5);    
    \coordinate (d9) at (1.8, 2.5);    
    \coordinate (d10) at (2.2, 2.5);    
    \coordinate (d11) at (3.8, 2.5);    
    \coordinate (d12) at (4.2, 2.5);

\filldraw[fill=black!20!white] (one) -- (two) -- (seven) -- (nine) -- cycle;
    
\draw[ultra thick] (one) -- (two);
\draw[thick] (one) -- (six);
\draw[ultra thick] (one) -- (fifteen);
\draw[thick] (one) -- (twentyfour);
\draw[thick] (two) -- (five);
\draw[ultra thick] (two) -- (sixteen);
\draw[thick] (two) -- (twentyfive);
\draw[thick] (five) -- (nine);
\draw[thick] (six) -- (ten);
\draw[ultra thick] (seven) -- (nine);
\draw[thick] (seven) -- (thirteen);
\draw[thick] (seven) -- (eighteen);
\draw[ultra thick] (seven) -- (nineteen);
\draw[ultra thick] (eight) -- (ten);
\draw[thick] (eight) -- (fourteen);
\draw[ultra thick] (nine) -- (twentysix);
\draw[thick] (nine) -- (thirtysix);
\draw[ultra thick] (ten) -- (twentyseven);
\draw[ultra thick] (eleven) -- (twentytwo);
\draw[thick] (eleven) -- (thirty);
\draw[thick] (thirteen) -- (sixteen);
\draw[thick] (fourteen) -- (fifteen);
\draw[ultra thick] (sixteen) -- (fortyfive);
\draw[thick] (eighteen) -- (twentytwo);
\draw[thick] (nineteen) -- (thirty);
\draw[ultra thick] (twentytwo) -- (thirtyfour);
\draw[thick] (twentyfour) -- (twentysix);
\draw[thick] (twentyfive) -- (twentyseven);
\draw[thick] (thirtyfour) -- (thirtysix);
\draw[thick] (e1) -- (fortyfive);
\draw[thick] (e1) -- (nineteen);
\draw[thick] (e3) -- (fortyfive);
\draw[thick] (e2) -- (sixteen);
\draw[ultra thick] (e4) -- (e5);
\draw[ultra thick] (e4) -- (twentyseven);
\draw[thick] (e3) -- (e5);
\draw[thick] (e4) -- (e2);
\draw[thick] (fifteen) -- (e8);
\draw[thick] (e8) -- (e9);
\draw[thick] (e9) -- (e10);
\draw[thick] (e10) -- (e11);
\draw[ultra thick] (e9) -- (twentysix);
\draw[ultra thick] (e11) -- (e7);
\draw[ultra thick] (e7) -- (thirtyfour);
\draw[thick] (e6) -- (twentysix);
\draw[thick] (e6) -- (e7);

\draw[ultra thick] (a1) -- (a2);
\draw[ultra thick] (a5) -- (a6);
\draw[ultra thick] (a9) -- (a10);
\draw[ultra thick] (a11) -- (a12);

\draw[ultra thick] (b1) -- (b2);
\draw[ultra thick] (b5) -- (b6);
\draw[ultra thick] (b9) -- (b10);
\draw[ultra thick] (b11) -- (b12);

\draw[ultra thick] (c1) -- (c2);
\draw[ultra thick] (c5) -- (c6);
\draw[ultra thick] (c9) -- (c10);
\draw[ultra thick] (c11) -- (c12);

\draw[ultra thick] (d1) -- (d2);
\draw[ultra thick] (d5) -- (d6);
\draw[ultra thick] (d9) -- (d10);
\draw[ultra thick] (d11) -- (d12);

    \filldraw[fill=white] (one) circle (2pt);
    \filldraw[fill=white] (two) circle (2pt);
    \filldraw[fill=white] (five) circle (2pt);
    \filldraw[fill=white] (six) circle (2pt);
    \filldraw[fill=white] (seven) circle (2pt);
    \filldraw[fill=white] (eight) circle (2pt);
    \filldraw[fill=white] (nine) circle (2pt) ;
    \filldraw[fill=white] (ten) circle (2pt) ;
    \filldraw[fill=white] (eleven) circle (2pt);
    \filldraw[fill=white] (thirteen) circle (2pt);
    \filldraw[fill=white] (fourteen) circle (2pt);
    \filldraw[fill=white] (fifteen) circle (2pt) ;
    \filldraw[fill=white] (sixteen) circle (2pt) ;
    \filldraw[fill=white] (eighteen) circle (2pt);
    \filldraw[fill=white] (nineteen) circle (2pt);
    \filldraw[fill=white] (twentytwo) circle (2pt);
    \filldraw[fill=white] (twentyfour) circle (2pt);
    \filldraw[fill=white] (twentyfive) circle (2pt) ;
    \filldraw[fill=white] (twentysix) circle (2pt) ;
    \filldraw[fill=white] (twentyseven) circle (2pt);
    \filldraw[fill=white] (thirty) circle (2pt) ;
    \filldraw[fill=white] (thirtyfour) circle (2pt);
    \filldraw[fill=white] (thirtysix) circle (2pt);
    \filldraw[fill=white] (fortyfive) circle (2pt);
    \filldraw[fill=white] (e1) circle (2pt);
    \filldraw[fill=white] (e2) circle (2pt);
    \filldraw[fill=white] (e3) circle (2pt);
    \filldraw[fill=white] (e4) circle (2pt);
    \filldraw[fill=white] (e5) circle (2pt);
    \filldraw[fill=white] (e6) circle (2pt);
    \filldraw[fill=white] (e7) circle (2pt);
    \filldraw[fill=white] (e8) circle (2pt);
    \filldraw[fill=white] (e9) circle (2pt);
    \filldraw[fill=white] (e10) circle (2pt);
    \filldraw[fill=white] (e11) circle (2pt);
    \filldraw[fill=white] (a1) circle (2pt);
    \filldraw[fill=white] (a2) circle (2pt);
    \filldraw[fill=white] (a5) circle (2pt);
    \filldraw[fill=white] (a6) circle (2pt);
    \filldraw[fill=white] (a9) circle (2pt);
    \filldraw[fill=white] (a10) circle (2pt);
    \filldraw[fill=white] (a11) circle (2pt);
    \filldraw[fill=white] (a12) circle (2pt);
    \filldraw[fill=white] (b1) circle (2pt);
    \filldraw[fill=white] (b2) circle (2pt);
    \filldraw[fill=white] (b5) circle (2pt);
    \filldraw[fill=white] (b6) circle (2pt);
    \filldraw[fill=white] (b9) circle (2pt);
    \filldraw[fill=white] (b10) circle (2pt);
    \filldraw[fill=white] (b11) circle (2pt);
    \filldraw[fill=white] (b12) circle (2pt);
    \filldraw[fill=white] (c1) circle (2pt);
    \filldraw[fill=white] (c2) circle (2pt);
    \filldraw[fill=white] (c5) circle (2pt);
    \filldraw[fill=white] (c6) circle (2pt);
    \filldraw[fill=white] (c9) circle (2pt);
    \filldraw[fill=white] (c10) circle (2pt);
    \filldraw[fill=white] (c11) circle (2pt);
    \filldraw[fill=white] (c12) circle (2pt);
    \filldraw[fill=white] (d1) circle (2pt);
    \filldraw[fill=white] (d2) circle (2pt);
    \filldraw[fill=white] (d5) circle (2pt);
    \filldraw[fill=white] (d6) circle (2pt);
    \filldraw[fill=white] (d9) circle (2pt);
    \filldraw[fill=white] (d10) circle (2pt);
    \filldraw[fill=white] (d11) circle (2pt);
    \filldraw[fill=white] (d12) circle (2pt);

\end{tikzpicture}
}
\caption{\label{Bi(23,31)} Partial Coxeter diagrams of the reflection subgroups of the quasi-reflective Bianchi groups $\widehat{Bi}(23)$ and $\widehat{Bi}(31)$. (Broken line branches intentionally omitted.)}
\end{figure}

\section{Hermitian model for the hyperbolic 3-space, Bianchi groups and Coxeter polyhedra}\label{background}
Consider the space $H_2$ of second-order Hermitian matrices and define a quadratic form $f$ on $H_2$ by the formula $f(x) = -2 \,\text{det}\, x$. The quadratic form $f$ has signature $(3,1)$, therefore it defines on $H_2$ the structure of Lorentzian $4$-space. Let $H_2^+$ denote the cone of positive definite matrices that are in one of the two connected components of the cone of all $x\in H_2$ with $f(x)<0$. The hyperbolic $3$-space $\Hy^3$ can be represented as the quotient $H_2^+ / \mathbb{R}_+$, where the group of positive real numbers $\mathbb{R}_+$ acts on $H_2$ by homotheties.

The transformations
\begin{equation}\label{transformation}
g(x) = \frac{1}{|\text{det } g|} gxg^* \; (g \in \mathrm{GL}_2(\mathbb{C})),
\end{equation}
where $^*$ denotes the Hermitian transpose, are pseudo-orthogonal transformations of the space $H_2$ that preserve the cone $H_2^+$. The orientation preserving motions of $\Hy^3$ are induced by these transformations $g$, and the orientation reversing motions are induced by compositions of $g$ with the complex conjugation $\tau$.  Therefore, the group of isometries of the hyperbolic 3-space in this model is the group $\mathrm{PGL}_2(\mathbb{C}) \rtimes \langle \tau \rangle$, and furthermore its discrete subgroups $Bi(m)$ and $\widehat{Bi}(m)$ are discrete groups of isometries of $\Hy^3$.

Under the action on the space $H_2$ the group $Bi(m)$ preserves the lattice $L_m$ which consists of the matrices with the entries in $\Am$. Let $\mathrm{O}_0(L_m)$ be the group of all pseudo-orthogonal transformations of the space $H_2$ that preserve the lattice $L_m$ and the cone $H_2^+$. It is an arithmetic subgroup of $\isom$, and Vinberg showed that in fact $\mathrm{O}_0(L_m) = \widehat{Bi}(m)$ \cite{vinberg90}. This implies in particular that the groups $\widehat{Bi}(m)$ and $Bi(m)$ have finite covolume.

Following Shaiheev \cite{shaiheev90}, we can choose a basis of $H_2$ in which the elements $x \in L_m$ are given by
\begin{equation}\label{basis}
  x = \left\{
  \begin{array}{ll}
    \begin{pmatrix}
      x_1 & x_3 - \sqrt{-m}x_4 \\
      x_3 + \sqrt{-m}x_4 & x_2
    \end{pmatrix} &  \; \text{ if } m \equiv 1,2 \text{ (mod 4)},\\
    \begin{pmatrix}
      x_1 & x_3 + \frac{1-\sqrt{-m}}{2}x_4 \\
      x_3 + \frac{1+\sqrt{-m}}{2}x_4 & x_2
    \end{pmatrix} &  \; \text{ if } m \equiv 3 \text{ (mod 4)},
   \end{array}\right.
\end{equation}
where $x_i \in \mathbb{Z}$. We see that, in these coordinates, $f$ is written as
\begin{equation}\label{form_f}
  f = \left\{
    \begin{array}{ll}
    -2x_1x_2 + 2x_3^2 + 2mx_4^2 &  \;  \text{ if } m \equiv 1,2 \text{ (mod 4)},\\
    -2x_1x_2 + 2x_3^2 + 2x_3x_4 + \frac{m+1}{2}x_4^2 &  \;  \text{ if } m \equiv 3 \text{ (mod 4)}.
   \end{array}\right.
\end{equation}

In our model of $\Hy^3$, a hyperplane is given by the set of rays in $H_2^+$ which are orthogonal to a vector $e$ of positive length. A hyperplane $\Pi_e$ defines two \emph{halfspaces}, $\Pi_e^+$ and $\Pi_e^-$, where $'\pm'$ is the sign of $(e,x)$ for $x$ in the corresponding halfspace, and a \emph{reflection}
\begin{equation*}
R_e: x \to x - 2\frac{(e,x)}{(e,e)}e,
\end{equation*}
where the inner product $(u,v) = \frac{1}{2}(f(u+v)-f(u)-f(v))$ is induced by $f$.
For brevity, a hyperplane associated to a vector $e_i$ will be denoted $\Pi_i$.

The vector $e$ which correspond to the reflection $R_e$ is defined up to scaling. If $e$ has rational coordinates we can normalise it so that the coordinates are coprime integers. With this normalisation we can assign to $R_e$ a correctly defined number $k = (e,e)$ and call $R_e$ a \emph{$k$-reflection}. Note that $k$ represents the spinor norm of $R_e$ (cf. \cite[p.~160]{EGM} for further discussion).

As $\widehat{Bi}(m) = \mathrm{O}_0(L_m)$ it follows that each reflection $R_e$ in $\widehat{Bi}(m)$ must satisfy the \emph{crystallographic condition:}
\begin{equation}\label{crystcondition}
\frac{2(e, x)}{(e, e)} \in \mathbb{Z}\ \text{ for any } x \in L_m.
\end{equation}

\medskip

An algorithm developed by Vinberg in \cite{vinberg72(2)} allows to construct a fundamental polyhedron of the maximal hyperbolic reflection subgroup of the integral automorphism group of a quadratic form. Let us briefly recall the procedure.

We begin by considering the stabiliser subgroup of a vector $u_0 \in L_m$ which corresponds to a point $x_0\in \overline{\Hy^3}$ (note that we allow $x_0\in\partial\Hy^3$). The \emph{polyhedral angle} at $x_0$ is defined by
\[P_0 = \bigcap_{i=1}^k \Pi_i^-,\]
with all the halfspaces being \emph{essential} (not wholly contained within another halfspace). There is a unique fundamental polyhedron $P$ of the reflection subgroup which sits inside $P_0$ and has $x_0$ as a vertex.

The algorithm continues by constructing further $\Pi_i$ such that
\[P = \bigcap_{i} \Pi_i^-.\]
This is done by choosing $e_i\in L_m$ such that $(e_i, e_i) > 0$, $(e_i, e_j) \le 0$ for all $j < i$, and the distance between $x_0$ and $\Pi_i$ is the smallest possible. The latter condition implies that all the $\Pi_i$s are essential.

We recall that the distance $\dist(x_0,\Pi_i)$ is related to the following function:
\begin{equation*}
\rho(x_0,\Pi_i) = \left|\frac{(e_i,u_0)}{\sqrt{(e_i, e_i)}}\right|.
\end{equation*}
Indeed, if $x_0\in\Hy^3$, we have $\sinh(\dist(x_0,\Pi_i)) = \rho(x_0,\Pi_i)/\sqrt{-(u_0, u_0)}$.
We will often refer to $\rho(x_0,\Pi_i)$ as the \emph{weight} of the plane $\Pi_i$ (and the vector $e_i$) with respect to $x_0$.

The algorithm terminates if it generates a configuration $P$  that has finite volume, in which case $\widehat{Bi}(m)$ is reflective. The corresponding Bianchi group is reflective if all the reflections produced by the algorithm that are not in $Bi(m)$ generate a finite subgroup of isometries (cf. Lemma~\ref{lem:sha}).

\medskip

We shall illustrate the algorithm by the following lemma. Let us fix $u_0 = (1,0,0,0)$. This vector is isotropic, so its stabiliser subgroup preserves a point at infinity. If $m\neq 3$, the corresponding stabiliser subgroup consists of the reflections in hyperplanes defined by the following vectors (cf.~\cite{shaiheev90}):
\begin{table}[H]
\centering
\begin{tabular}{l|l}
  $e_1$ & $(0, 0, -1, 0)$\\
  $e_2$ & $(1, 0, 1, 0)$\\
  $e_3$ & $(0, 0, 0, -1) \text{ for } m \equiv 1, 2\text{ (mod 4)}; \text{ or } (0, 0, 1, -2) \text{ for } m \equiv 3\text{ (mod 4)}$\\
  $e_4$ & $(m, 0, 0, 1) \text{ for } m \equiv 1, 2\text{ (mod 4)}; \text{ or } (m, 0, -1, 2) \text{ for } m \equiv 3\text{ (mod 4)}$
\end{tabular}
\end{table}

In the group $Bi(3)$ the fundamental domain of the reflection subgroup has four faces (see Figure~\ref{fundamentaldomains}). The normal vectors corresponding to these reflections can be found in \cite{shaiheev90} and are different from those listed above. As this is a configuration with finite volume the algorithm terminates, and hence we exclude $m=3$. We now look for the next vector produced by Vinberg's algorithm in all other cases.

\begin{lemma}\label{e5inlattice}
For every $m \neq 3$, we have $e_5 = (-1, 1, 0, 0)$.
\end{lemma}
\begin{proof}
Assume that $m \equiv 1, 2\text{ (mod 4)}$. We know the first four vectors in $L_m$, and that all subsequent vectors must have non-positive inner product with them, so we have four inequalities which constrain the coefficients of the remaining vectors. Let $\textbf{x}= (x_1,x_2,x_3,x_4)$ be the first vector that is to be found by the algorithm. The inequalities can be summarised as follows:
\begin{eqnarray*}
x_2 \ge 2x_3 \ge 0,\\
mx_2 \ge 2mx_4 \ge 0.
\end{eqnarray*}

The weight function $\rho$ of $\textbf{x}$ is given by
\begin{equation*}
\rho(u_0, \textbf{x}) = \frac{x_2}{\sqrt{(\textbf{x},\textbf{x})}},
\end{equation*}
which we want to minimise, so we can try choosing $x_2$ as small as possible. If $x_2 = 0$ then by the above inequalities we recover the isotropic vector $u_0$ (up to a scalar multiple), so we choose $x_2 = 1$, and hence $x_3 = x_4 = 0$. Now $(\textbf{x}, \textbf{x}) = -2x_1$, so $x_1$ must be negative, and by considering the crystallographic condition with respect to $e_2$ we can conclude that $x_1 = -1$. Therefore, our vector $\textbf{x}$ has length 2 and $\rho(u_0, \textbf{x}) = \frac{1}{\sqrt{2}}$.

That this is actually minimal can be confirmed by considering the crystallographic conditions associated to the vectors $e_1$ and $e_2$:
\begin{equation*}
\frac{2(\textbf{x}, e_1)}{(\textbf{x},\textbf{x})} = \frac{-2(x_4 + 2x_3)}{(\textbf{x},\textbf{x})} \in \mathbb{Z}; \; \frac{2(\textbf{x}, e_2)}{(\textbf{x},\textbf{x})} = \frac{-2x_2 + 2(x_4 + 2x_3)}{(\textbf{x},\textbf{x})} \in \mathbb{Z},
\end{equation*}
which imply that $|(\textbf{x},\textbf{x})| \le |2x_2|$. We are therefore searching for a solution to the following inequality
\begin{equation*}
\frac{x_2}{\sqrt{(\textbf{x},\textbf{x})}} \le \frac{1}{\sqrt{2}},
\end{equation*}
which, given that $x_2$ is strictly positive, implies that $x_2 = 1$.

The case $m \equiv 3\text{ (mod 4)}$, $m\neq 3$, is treated similarly and we skip the details.
\end{proof}

If the algorithm terminates then $P$ is an acute-angled polyhedron, and in fact it is a \emph{Coxeter polyhedron}. Recall that an acute-angled polyhedron is called a Coxeter polyhedron if all the dihedral angles at the intersections of pairs of faces are integer submultiples of $\pi$ (or zero). A complete presentation of an acute-angled polyhedron is given by a \emph{Gram matrix}. A Gram matrix $G = (g_{ij})$ is a symmetric matrix with entries:
\begin{equation*}
  g_{ij} = \left\{
  \begin{array}{ll}
    1 &  \; \text{ if } \; i = j,\\
    -\cos(\frac{\pi}{n}) &  \; \text{ if } \; \angle (\Pi_i, \Pi_j) = \frac{\pi}{n},\\
    -1 &  \; \text{ if } \; \angle (\Pi_i, \Pi_j) = 0,\\
    -\cosh(\dist(\Pi_i, \Pi_j)) &  \; \text{ if $\Pi_i$ and $\Pi_j$ do not intersect,}
   \end{array}\right.
\end{equation*}
where $\dist(\Pi_i, \Pi_j)$ is the minimal hyperbolic distance between the two hyperplanes. The entries of the Gram matrix may be computed directly from the normal vectors to the hyperplanes $\Pi_i$ as
\[g_{ij}=\frac{(e_i, e_j)}{\sqrt{(e_i, e_i)(e_j,e_j)}}.\]

Another presentation of an acute-angled polyhedron which we are going to use is the \emph{Coxeter diagram}. It is a graph which reproduces most of the information in the Gram matrix, with the exception of the distances between non-intersecting planes. Each vertex of a Coxeter diagram corresponds to a hyperplane, and the edges are as presented in Table~\ref{coxeteredges}. We will label the vertices of a diagram by the vectors which define the hyperplanes (or their numbers).
\begin{table}[H]
\caption{\label{coxeteredges} The edges of a Coxeter diagram}
\centering
\begin{tabular}{l|l}
  Type of edge & Corresponds to\\
  \hline
  comprised of $m-2$ lines, or labelled $m$ & a dihedral angle $\frac{\pi}{m}$\\
  a single heavy line & a ``cusp'', or a dihedral angle zero\\
  a dashed line & two divergent faces\\
  no line & a dihedral angle $\frac{\pi}{2}$
\end{tabular}
\end{table}

Gram matrices (and, by association, Coxeter diagrams) have a well known classification according to their determinant and definiteness (cf. \cite{vinberg93}). An important part of the classification here is the parabolic and elliptic subdiagrams. These correspond to ideal vertices and regular vertices of the fundamental domain respectively. The \emph{rank} of a Coxeter diagram is equal to the rank of its Gram matrix. For elliptic diagrams, the rank is equal to the number of vertices, while in the parabolic case the rank is given by the number of vertices minus the number of connected components.

A ``cusp'' as presented in Table~\ref{coxeteredges} is comprised of only two hyperplanes, and shall be called a \emph{cusp pair}. In dimension $3$ two such cusp pairs, with the constituent hyperplanes mutually orthogonal, form a parabolic subdiagram of rank $2$, or in other words a properly formed cusp of the fundamental domain.

In order to determine whether a polyhedron of finite volume has been generated by the algorithm, we use the following propositions. These statements were proved for the Gram matrix of a Coxeter polyhedron of finite volume by Vinberg \cite{vinberg85}, who noted that the results may be reinterpreted in terms of the Coxeter diagram. The form given here is due to Bugaenko \cite{bugaenko92}.

\begin{proposition}\label{bugaenko1} 
A Coxeter polyhedron is bounded if and only if any elliptic subdiagram of rank $n-1$ of its Coxeter diagram can be extended to an elliptic subdiagram of rank $n$ in precisely two ways.
\end{proposition}

\begin{proposition}\label{bugaenko2} 
A Coxeter polyhedron is of finite volume if and only if any elliptic subdiagram of rank $n-1$ of its Coxeter diagram can be extended to an elliptic subdiagram of rank $n$ or a parabolic subdiagram of rank $n-1$ in precisely two ways.
\end{proposition}

Geometrically these statements mean that each edge of the polyhedron has two vertices, either one or both of which may be at the ideal boundary of the hyperbolic space.

\section{Reflection subgroups of Bianchi groups}\label{reflbackground}
From the previous section we know the algorithm for finding reflective isometries in Bianchi groups $Bi(m)$ and $\widehat{Bi}(m)$. We will now discuss the structural information about these groups which is provided by the existence of these reflections. A part of the discussion applies in much wider generality, so in the beginning we will consider discrete groups of isometries of an $n$-dimensional hyperbolic space $\Hy^n$ and later restrict to the case of $n=3$.

Let $\Gamma$ be a discrete subgroup of $\mathrm{Isom}(\Hy^n)$, and let $\Gamma_r$ be its subgroup generated by all the reflections from $\Gamma$. Since a conjugation of a reflection in $\mathrm{Isom}(\Hy^n)$ is again a reflection, the subgroup $\Gamma_r$ is normal in $\Gamma$ and we have
\begin{equation}\label{Gamma_rtimesH}
\Gamma = \Gamma_r \rtimes H.
\end{equation}
The group $\Gamma_r$ has a fundamental domain which is a polyhedron $P$ (which may have infinitely many faces, and may also have infinite volume) in $\Hy^n$, whose facets are precisely the mirror hyperplanes of the hyperbolic reflections which generate $\Gamma_r$. From now on by \emph{fundamental polyhedron} of $\Gamma_r$ we will always mean this Coxeter polyhedron. The group $H$ in decomposition (\ref{Gamma_rtimesH}) can be identified with the symmetry group of $P$. This fact was proved in \cite{vinberg67} for the case when the group $H$ is finite, but the same argument works in general (see also \cite[Lemma 5.2]{agolbelolipetskystormwhyte08} where Vinberg's proof is repeated).

\begin{definition}
A subgroup $\Gamma \subset \mathrm{Isom}(\Hy^n)$ is called a \emph{lattice} if it is a discrete subgroup of finite covolume.

A lattice $\Gamma$ is called \emph{reflective} if its non-reflective part $H$ in (\ref{Gamma_rtimesH}) is finite, and \emph{quasi-reflective} if $H$ is infinite, has an infinitely distant fixed point $q \in \partial \Hy^n$, and leaves invariant a horosphere $S=\mathbb{S}^{n-1}$ of the maximal dimension with the centre at $q$.
\end{definition}

From the definition it follows that quasi-reflective lattices are necessarily non-cocompact (which is clearly not the case for the reflective ones). The group $H$ acts by affine isometries of $S$ and is itself a discrete subgroup of $\text{Aff}(S)$. It has a finite index subgroup $H_t$ generated by translations of $S$ (cf. \cite[Section 4.2]{humphreys90}). The rank $r$ of $H_t$ is equal to the number of the linearly independent translations. We will call it the \emph{quasi-reflective rank} of $\Gamma$ and we will also say that $\Gamma$ is a quasi-reflective group of rank $r$. 

When $H$ is considered as the symmetry group of $P$, and hence a subgroup of $\mathrm{Isom}(\Hy^n)$, the elements from $H_t$ correspond to the parabolic transformations. Let $\Gamma_1 \le \Gamma$ be a subgroup defined by $\Gamma_1 = \Gamma_r \rtimes H_t$.  It is easy to see that $\Gamma_1$ is a normal subgroup of $\Gamma$. We denote the corresponding quotient group by $Q$. The whole configuration is represented by the following diagram:
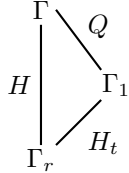
\begin{figure}[H]
  \centering
  \begin{tikzpicture}
    \coordinate (one) at (0, 0);
    \coordinate (onelabel) at (0, -1);
    \coordinate (one2) at (0.2, 0);
    \coordinate (one3) at (0, -0.2);
    \coordinate (two) at (0, -2);
    \coordinate (twolabel) at (0.5, -1.5);
    \coordinate (two2) at (0, -1.8);
    \coordinate (two3) at (0.2, -1.8);
    \coordinate (three) at (1, -1);
    \coordinate (threelabel) at (0.5, -0.5);
    \coordinate (three2) at (0.8, -0.8);
    \coordinate (three3) at (0.8, -1.2);

    \draw[thick] (one3) -- (two2);
    \draw[thick] (one2) -- (three2);
    \draw[thick] (two3) -- (three3);

    \node at (one) {$\Gamma$};
    \node[left] at (onelabel) {$H$};
    \node at (two) {$\Gamma_r$};
    \node[below right] at (twolabel) {$H_t$};
    \node at (three) {$\Gamma_1$};
    \node[above right] at (threelabel) {$Q$};
  \end{tikzpicture}
\end{figure}

The fundamental polyhedron $P$ of the reflection subgroup of a quasi-reflective group $\Gamma$ is an infinite volume polyhedron in $\Hy^n$ with infinitely many facets. Its symmetry group $H$ is isomorphic to an affine crystallographic group of the translation rank $r \le n-1$ and $P / H$ has finite volume. Following Ruzmanov \cite{ruzmanov90} we will call such polyhedra \emph{quasi-bounded}. A quasi-bounded polyhedron $P$ has an infinitely distant point $q$ such that the intersection of some horosphere with the centre at $q$ and $P$ is unbounded. This point $q$ is unique and it is called the \emph{singular point} of $P$.

The relationship between reflective and quasi-reflective lattices is demonstrated by the following lemma.

\begin{lemma}\label{reflorq}
If $\Gamma \le \Gamma^\prime$ are two lattices in $\mathrm{Isom}(\Hy^n)$ and $\Gamma$ is quasi-reflective, then $\Gamma^\prime$ is either reflective or quasi-reflective.
\end{lemma}
\begin{proof}
Let $\Gamma = \Gamma_r \rtimes H$ and $P$ is a quasi-bounded fundamental polyhedron of $\Gamma_r$. Let $q$ denote the singular point of $P$. Let $\Gamma_r^\prime$ be the maximal reflective subgroup of $\Gamma^\prime$.

As $\Gamma \le \Gamma^\prime$, we have $\Gamma_r \le \Gamma_r^\prime$.

If $\Gamma_r^\prime$ has finite covolume, then $\Gamma^\prime / \Gamma_r^\prime$ is finite and $\Gamma^\prime$ is reflective. Assume that $covol(\Gamma_r^\prime) = \infty$ and let $P^\prime$ denote the quotient polyhedron $\Hy^n / \Gamma_r^\prime$. We have $P = \bigcup_{\gamma \in \Gamma_r^\prime / \Gamma_r}\gamma(P^\prime)$ a quasi-bounded polyhedron and $P^\prime$ has infinite volume, hence the horosphere with centre $q$, which has unbounded intersection with $P$, has also unbounded intersection with $P^\prime$. On the other hand, if a horosphere has unbounded intersection with $P^\prime$ then it obviously has also unbounded intersection with $P$. It follows that $q$ is a unique singular point of $P^\prime$ and $P^\prime$ is a quasi-bounded polyhedron. Hence $\Gamma^\prime$ is a quasi-reflective group.
\end{proof}

From now on let us restrict to the case $n= 3$. In \cite{vinberg90}, Vinberg showed that reflectivity of $Bi(m)$ or $\widehat{Bi}(m)$ imposes a strong restriction on the class number of the field $K_m = \mathbb{Q}[\sqrt{-m}]$. We will recall the main steps of his argument and extend it to the quasi-reflective groups.

The group $Bi(m)$ acts on the space $H_2$ preserving the lattice $L_m$, which consists of the matrices with entries in $\Am$ (see Section~\ref{background}). A point at infinity of the space $\Hy^3$ is called \emph{rational} if the corresponding one-dimensional subspace of $H_2$ contains a non-zero element of the lattice $L_m$. The rationality of a point $p\in \partial \Hy^3$ is equivalent to the fact that $p\in \gamma \bar{D}$ for some $\gamma \in \Gamma$, where $\Gamma$ is a subgroup of $\mathrm{Isom}(\Hy^3)$ commensurable with $Bi(m)$ and $D$ is a convex fundamental domain of $\Gamma$.

From the other side, there is an isomorphism $\partial\Hy^3 \simeq P\mathbb{C}^2$ which maps the set of the rational points to the set $PK_m^2 \subset P\mathbb{C}^2$. Using the theory of modules over Dedekind rings, this leads to a natural one-to-one correspondence between the set $PK_m^2 / \mathrm{PGL}_2(\Am)$ and the group of ideal classes $C(\Am)$. Hence the number of cusps of $\Hy^3 / Bi(m)$ is equal to the order of $C(\Am)$. This result was first observed by Bianchi and was proved in greater generality by Hurwitz in a letter to Bianchi \cite[p.~103--105]{bianchi1892}

We now look at the action of the group $\widehat{Bi}(m)$ on the set of the rational points at infinity. It induces the action of the quotient group $\widehat{Bi}(m) / \mathrm{PGL}_2(\Am) = C_2(\Am) \rtimes \langle\tau\rangle$, where $C_2(\Am)$ is the 2-periodic part of $C(\Am)$, on the quotient set $PK_m^2 / \mathrm{PGL}_2(\Am)$, and hence on the group $C(\Am)$.

Following \cite[Section 3]{vinberg90}, this action $\circ$ is defined by
\begin{align*}
\tau \circ [\mathfrak{b}] & = [\mathfrak{b}]^{-1}; \\
[\mathfrak{a}] \circ [\mathfrak{b}] & = [\mathfrak{a}][\mathfrak{b}],
\end{align*}
where $[\mathfrak{b}] \in C(\Am)$, $[\mathfrak{a}] \in C_2(\Am)$.

If a discrete group $\Gamma$ commensurable with $Bi(m)$ is reflective, then the stabiliser of any rational point at infinity in the group $\Gamma$ contains a reflection. If it is quasi-reflective then the stabiliser of any rational point at infinity not $\mathrm{PGL}_2(\Am)$-equivalent to the singular point contains a reflection. Hence the stabiliser of any element of the group $C(\Am)$ (respectively, of $C(\Am) \backslash [\mathfrak{c}], \tau\circ [\mathfrak{c}] = [\mathfrak{c}]^{-1}$, for some $[\mathfrak{c}] \in C(\Am)$) under the action described above contains an element of the group $C_2(\Am) \rtimes \langle\tau\rangle$ not belonging to $C_2(\Am)$ and arising from some element of $\Gamma$. Note that, if in the quasi-reflective case the singular point has no reflections in its stabiliser, then $[\mathfrak{c}]$ can be not an involution in the group $C(\Am)$. This case corresponds to the groups of the quasi-reflective rank $2$. All the other cases are identical to the ones considered in \cite{vinberg90}.

We have:

\begin{enumerate}
\item \label{qrrank1ar} If the group $Bi(m)$ is reflective or quasi-reflective of rank $1$, then the element $\tau$ must act trivially on $C(\Am)$. Hence the orders of all elements of $C(\Am)$ divide 2.
\item \label{qrrank2ar} If the group $\widehat{Bi}(m)$ is reflective or quasi-reflective of rank $1$, then the stabiliser of any $[\mathfrak{b}] \in C(\Am)$ must contain an element of the form $[\mathfrak{a}]\tau$, where $[\mathfrak{a}] \in C_2(\Am)$, and we have
\begin{align*}
[\mathfrak{a}] \tau \circ [\mathfrak{b}] & = [\mathfrak{a}][\mathfrak{b}]^{-1} = [\mathfrak{b}]\\
[\mathfrak{b}]^2 & = [\mathfrak{a}]\\
[\mathfrak{b}]^4 & = 1
\end{align*}
\item \label{qrrank3ar} If the group $Bi(m)$ is quasi-reflective of rank $2$, then the argument of case (\ref{qrrank1ar}) applies to all elements of the group $C(\Am)$ except possibly for some $[\mathfrak{c}]$ and $[\mathfrak{c}]^{-1}$, which can be not involutive. The only possibilities to achieve the latter is that $[\mathfrak{c}]$ has order $3$ or $4$ in $C(\Am)$.
\item If the group $\widehat{Bi}(m)$ is quasi-reflective of rank $2$, then the description of $C(\Am)$ is obtained as a combination of cases (\ref{qrrank2ar}) and (\ref{qrrank3ar}).
\end{enumerate}

All together these lead to the following statement.

\begin{proposition}\label{classfilter}
The class groups of the fields $K_m$ satisfy:
\begin{enumerate}
\item If $Bi(m)$ is reflective or quasi-reflective of rank $1$ then $C(\Am) \cong (\mathbb{Z} / 2\mathbb{Z})^n, n \in \mathbb{Z}_{\geq 0}$;
\item If $\widehat{Bi}(m)$ is reflective or quasi-reflective of rank $1$ then $C(\Am) \cong (\mathbb{Z} / 2\mathbb{Z})^n \times (\mathbb{Z} / 4\mathbb{Z})^l, n,l \in \mathbb{Z}_{\geq 0}$;
\item If $Bi(m)$ is quasi-reflective of rank $2$ then $C(\Am) \cong (\mathbb{Z} / 2\mathbb{Z})^n \times (\mathbb{Z} / q \mathbb{Z})^k, n \in \mathbb{Z}_{\geq 0}, k = 0 \text{ or } 1 \text{ and } q = 3 \text{ or } 4$;
\item If $\widehat{Bi}(m)$ is quasi-reflective of rank $2$ then $C(\Am) \cong (\mathbb{Z} / 2\mathbb{Z})^n \times (\mathbb{Z} / 3\mathbb{Z})^k \times (\mathbb{Z} / 4\mathbb{Z})^l, n,l \in \mathbb{Z}_{\geq 0}, k = 0 \text{ or } 1$.
\end{enumerate}
\end{proposition}

\section{Finiteness of quasi-reflective Bianchi Groups}\label{QRbackground}
We begin proving Theorem~\ref{thm3} by establishing an effective finiteness result for the quasi-reflective Bianchi groups. Our method is similar to \cite{belolipetsky09}, which in turn is a development of \cite{agol06} and which provides corresponding finiteness result for Theorems~\ref{thm1} and \ref{thm2}. An important extra ingredient of the proof in quasi-reflective case is provided by Lemma~\ref{lem1} below.

Let $\Gamma$ be a Bianchi group $Bi(m)$ or an extended Bianchi group $\widehat{Bi}(m)$. By Lemma~\ref{reflorq}, if $Bi(m)$ is quasi-reflective then $\widehat{Bi}(m)$ is reflective or quasi-reflective, and there is also a possibility that $\widehat{Bi}(m)$ is quasi-reflective while $Bi(m)$ is not. The method of this part of the proof applies equally well to both $Bi(m)$ and $\widehat{Bi}(m)$, and so we will not distinguish between these groups here. Addressing the differences between them requires a finer analysis which we will do later on.

Assume that $\Gamma$ is quasi-reflective.

We have $\Gamma = \Gamma_r \rtimes H$, $P$ is the quasi-bounded fundamental polyhedron of $\Gamma_r$, $H$ is an affine crystallographic group of rank $\le 2$ which acts by (hyperbolic) isometries on $P$, and $\mathcal{O} = \Hy^n / \Gamma = P / H$ is a finite volume hyperbolic 3-orbifold.

We will need a notion of (piecewise) \emph{conformal volume} for Riemannian orbifolds. This is a direct generalisation of the Li-Yau conformal invariant \cite{liyau82} which was developed in \cite{agol06} and \cite{agolbelolipetskystormwhyte08}. We refer to \cite[Section 2]{agolbelolipetskystormwhyte08} for the definition and basic facts about the conformal volume.

Let $\mathcal{Q} = \Hy^n/H$, an infinite volume hyperbolic 3-orbifold. Then $\mathcal{O} = \mathcal{Q} / \Gamma_r^\prime$, where $\Gamma_r^\prime \le \Gamma_r$ is a group generated by reflections. Geometrically $\mathcal{O}$ is obtained by adding to $\mathcal{Q}$ the mirror sides which correspond to the generators of $\Gamma_r^\prime$. Therefore $\mathcal{O}$ embeds into $\mathcal{Q}$ as a codimension-$0$ suborbifold. By \cite[Section 2, Fact 4]{agolbelolipetskystormwhyte08}, the piecewise conformal volume
\begin{equation*}
V_{PC}(3, \mathcal{O}) \le V_{PC}(3, \mathcal{Q}).
\end{equation*}

In order to estimate the volume of $V_{PC}(\mathcal{Q})$ we will use the following result.

\begin{lemma}\label{lem1}
Let $H$ be an affine crystallographic group acting on a Euclidean plane. Then there exists a planar affine crystallographic group $H_r$, which is generated by reflections, such that $H_r \ge H$ and $[H_r : H] \le 4$.
\end{lemma}
\begin{proof} It is well known that there are precisely $17$ isomorphism classes of the plane crystallographic groups (cf. \cite{conway92}). We will use Conway's notation for the corresponding $2$-dimensional affine orbifolds [loc. cit.]. The following covering relations take place:
\begin{itemize}
\item $\star 632$ is covered by $632$ with degree 2;
\item $\star 442$ is covered by $442$ and $4\star 2$, each with degree 2;
\item $\star 333$ is covered by $333$ and $3\star 3$, each with degree 2;
\item $\star 2222$ is covered by $2222$ and $2\star 22$, each with degree 2. In turn, $2 \star 22$ is covered by $22\star$ and $22\times$, each with degree 2;
\item $\star \star$ is covered by $\star\times$ and $\times\times$ each with degree 2, and $\times\times$ is $2$-covered by $\circ$.
\end{itemize}
The groups of $\star 632$, $\star 442$, $\star 333$, $\star 2222$, and $\star \star$  are generated by reflections, and all 17 affine 2-orbifolds cover them with degree at most 4.
\end{proof}

Returning to the proof of the theorem, we have $H \le H_r$ generated by (affine) reflections of the horosphere $S$ (if $H$ has rank $1$, then $H = \langle h \rangle \le H_r$ with index $2$, as it follows from the representation of the translation isometry $h$ as a product of two reflections). We can extend these reflections to the hyperbolic reflections in $\Hy^3$ and thus obtain $H \le H^\prime$, $H^\prime$ is generated by reflections in $\Hy^3$ and $[H^\prime : H] \le 4$. Now, Fact 1 (\cite{agolbelolipetskystormwhyte08}) implies
\begin{equation*}
V_{PC}(3, \mathcal{Q}) \le 4 \,V_{PC}(3, \mathcal{Q}^\prime),
\end{equation*}
where $\mathcal{Q}^\prime = \Hy^3/H^\prime$. Moreover, Facts 3 and 4 (\emph{ibid.}) give
\begin{equation*}
V_{PC}(3, \mathcal{Q}^\prime) = \text{Vol}(\mathbb{S}^3) = 2\pi^2,
\end{equation*}
because $\mathcal{Q}^\prime$ is a reflection orbifold.

All together we obtain
\begin{equation*}
V_{PC}(3, \mathcal{O}) \le 8\pi^2.
\end{equation*}

Accidentally, the same upper bound for the conformal volume was obtained in \cite{agol06} for the case of reflective Bianchi groups (as well as more general arithmetic Kleinian groups).

The rest of the argument is based on the Li-Yau inequality which relates the hyperbolic volume, conformal volume, and the first non-zero eigenvalue of the Laplace operator on $\mathcal{O}$. It works exactly as in \cite{agol06} and \cite{belolipetsky09}, and the numerical computation in \cite[Section 4.3]{belolipetsky09} implies that there are 882 values of $m$ which satisfy the conditions.

\begin{remark}
Note that if it would be possible to generalise Lemma~\ref{lem1} to higher dimensions, the argument presented above would give an alternative proof of the finiteness of quasi-reflective groups, which was previously established by Nikulin in \cite{nikulin96}, and would provide effective bounds for their invariants. More precisely, we are interested in good upper bounds for the conformal volumes of the quotients of $\mathbb{R}^n$ by the affine crystallographic groups. At the present moment we do not know how to prove such bounds. For example, it would be interesting to compute or estimate the conformal volume of the group of automorphisms of the Leech lattice.
\end{remark}

\section{Technical lemmas}\label{sec:lemmas}
In order to reduce the list of admissible groups we will consider the connection between the class number $h_m$ of $K_m = \mathbb{Q}[\sqrt{-m}]$ and the geometry of $\mathcal{O} = \Hy^3/Bi(m)$. To this end we will use the generalisation of the results of Vinberg and Shvartsman which was carried out in Section~\ref{reflbackground} and the technical lemmas from this section.

We begin with a lemma which was stated without a proof (and under an additional assumption that $m \equiv 3\text{ (mod 4)})$ in Shvartsman's paper \cite{shvartsman98}. Considering the importance of this result for our argument we present the details here. Recall from Section~\ref{background} the definition of a $k$-reflection and the fact that its spinor norm is represented by $k$.

\begin{lemma}\label{lem:refllengths} 
The subgroup $\Gamma_r < \Bi$ of reflections consists of only $2$- and $2m$-reflections, and all such reflections in $\hBi$ lie in $\Gamma_r$.
\end{lemma}
\begin{proof}
A reflection $\gamma$ can be written as a composition $[g]\tau$ where $g\in\mathrm{GL}_2(\C)$ acts as in \eqref{transformation} and $\tau$ is a complex conjugation. Assume (with an abuse of notation) that $g\in \PGL(\Am)$. Then since $\PGL(\Am)\ge\PSL(\Am)$ with index $2$, $g\in\PSL(\Am)$ or $g\in g_0\PSL(\Am)$, where we can take $g_0 = \left( \begin{smallmatrix} -1 & 0\\ 0&1 \end{smallmatrix}\right)$.

Let $V$ be a $4$-dimensional regular quadratic space over $\Q$ with a quadratic form $f$ defined by \eqref{form_f}. Following \cite[\S~11]{EGM} (see also \cite{jamesmaclachlan96}), we have an exact sequence
\begin{equation*}
1 \to \PSL(\Am) \stackrel{\Phi^*}{\to} \mathrm{SO}(V) \stackrel{\theta}{\to} \mathrm{cok} \to 1,
\end{equation*}
where $\Phi^*$ maps $g \in \PSL(\Am)$ to the corresponding transformation $[g]$ of the form \eqref{transformation} and $\theta$ is the spinor norm map $\mathrm{SO}(V) \to \Q^*/{\Q^*}^2$. By \cite[Theorem~2.1]{jamesmaclachlan96}, $\Phi^*(\PSL(\Am)) = \mathrm{O}'(L_m)$, where the lattice $L_m$ can be defined as in \eqref{basis} and $\mathrm{O}'(L_m)$ denotes the kernel of the spinor norm map of $\mathrm{SO}(L_m)$. It follows that for $\gamma = [g]\tau$ stabilising the lattice $L_m$, $g\in\PGL(\Am)$ if and only if $\theta([g]) = {\Q^*}^2$ or $\theta([g]) = \theta([g_0]){\Q^*}^2$. This in turn means that
\begin{equation}\label{l6.1:cond}
\theta(\gamma) = \theta(\tau){\Q^*}^2 \text{ or } \theta(\gamma) = \theta(\tau)\theta([g_0]){\Q^*}^2.
\end{equation}

To compute the spinor norms of $[g_0]$ and $\tau$ we need to consider two cases:

First assume that $m \equiv 1,2\text{ (mod 4)}$. Then in the basis of $V$ as in \eqref{basis}, $\tau$ corresponds to the reflection $R_e$ with $e = (0,0,0,1)$. Its spinor norm is $f(e,e){\Q^*}^2 = 2m{\Q^*}^2$. The transformation $[g_0]$ maps $(x_1,x_2,x_3,x_4)$ to $(x_1,x_2,-x_3,-x_4)$. It can be written as a product of two reflections $R_u$ and $R_v$ with $u = (0,0,1,0)$ and $v = (0,0,0,1)$, and hence $\theta([g_0]) = f(u,u)f(v,v){\Q^*}^2 = m{\Q^*}^2$.

In the second case $m \equiv 3\text{ (mod 4)}$. In the basis \eqref{basis}, $\tau$ corresponds to the reflection $R_e$ with $e = (0,0,-1,2)$, $\theta(\tau) = 2m{\Q^*}^2$. The transformation $[g_0]$ again maps $(x_1,x_2,x_3,x_4)$ to $(x_1,x_2,-x_3,-x_4)$. It can be written as a product of $R_u$ and $R_v$ with $u = (0,0,1,0)$ and $v = (0,0,-1,2)$, and $\theta([g_0]) = m{\Q^*}^2$.

\smallskip

We now come back to the formula \eqref{l6.1:cond} to conclude that a reflection $\gamma \in \hBi$ is in $\Bi$ if and only if $\theta(\gamma) = 2m{\Q^*}^2$ or $2{\Q^*}^2$ (where the first case corresponds to $g \in \PSL(\Am)$ and the second to $g \not\in \PSL(\Am)$).
\end{proof}

The next lemma was proved in \cite{shaiheev90}. Here we give another argument for the same result based on Lemma~\ref{lem:refllengths}.

\begin{lemma}\label{lem:sha}
Assume that $\hBi$ is reflective and $S$ is a (minimal) set of generating reflections of $\hBi_r$. Let $H$ be the subgroup of $\hBi$ generated by all reflections in $S$ that are not in $\Bi$. Then $\Bi$ is reflective if and only if $H$ is finite.
\end{lemma}

\begin{proof}
By Lemma~\ref{lem:refllengths} the reflection subgroup $\Bi_r$ is generated by all the $2$- and $2m$-reflections from $\hBi_r$. The spinor norm is invariant under conjugation by the elements with an integer spinor norm, hence this set is closed under conjugation by the elements of $\hBi_r$ and $\Bi_r \lhd \hBi_r$. By the definition of $H$ we have $H \cong \hBi_r/\Bi_r$. Now the group $\hBi$ is reflective implies $\hBi/\hBi_r$ is finite, therefore $\Bi/\Bi_r$ is finite if and only if $H$ is a finite group.
\end{proof}

As was already mentioned in Section~\ref{reflbackground}, by a result of Bianchi and Hurwitz the number of cusps of $\Hy^3/\Bi$ is given by the class number $h_m$ of the field $K_m$. This leads to the following property of the reflective and quasi-reflective Bianchi groups which can be checked algorithmically.

Recall that $\hBi/\Bi \cong C_2(O_m)$, the $2$-periodic part of the class group of $K_m$, whose order is given by
$$
h_{2,m} = \left\{
  \begin{array}{ll}
    2^t     &  \; \text{ for } m = 1\:(\text{mod }4),\\
    2^{t-1} &  \; \text{ for } m = 2,3\:(\text{mod }4),\\
   \end{array}\right.
$$
where $t$ denotes the number of the prime divisors of $m$.

\begin{proposition}\label{QR-cuspcnt}
Let $\Gamma$ be a lattice in $\isom$ and $\Gamma_r$ its subgroup generated by (all) reflections.

For $\Gamma$ being reflective it is necessary that
\begin{itemize}
\item[(1)] If $\Gamma = \Bi$ then $\Hy^3/\Gamma_r$ has at most $12h_m$ cusps;
\item[(2)] If $\Gamma = \hBi$ then $\Hy^3/\Gamma_r$ has at most $12h_mh_{2,m}$ cusps.
\end{itemize}
For $\Gamma$ being quasi-reflective, let $v$ be a vertex of the Coxeter diagram of $\Gamma_r$ such that the reflection hyperplane corresponding to $v$ does not pass through the singular point at infinity. The necessary condition are
\begin{itemize}
\item[(3)] If $\Gamma = \Bi$ then $v$ is adjacent to at most $12(h_m-1)$ cusp pairs;
\item[(4)] If $\Gamma = \hBi$ then $v$ is adjacent to at most $12h_{2,m}(h_m-1)$ cusp pairs.
\end{itemize}
\end{proposition}

\begin{proof}
By a lemma of Vinberg (see Section~\ref{QRbackground}), in both reflective and quasi-reflective cases the quotient group $\Gamma/\Gamma_r$ is isomorphic to a (sub)group of the group of symmetries of the fundamental Coxeter polyhedron $P$ of $\Gamma_r$ which, moreover, can be identified with a subgroup of $\Gamma$. Note that this result applies also when $\Gamma = \Bi$, in which case $\Gamma/\Gamma_r$ might be isomorphic to a proper subgroup of
$\sym$.

Finite subgroups of Bianchi groups were classified by Klein in \cite{klein1875}. They are the trivial group, $\Z/2\Z$, $\Z/3\Z$, $\Z/2\Z \times \Z/2\Z$, $S_3$ and $A_4$, and the maximal order is $12$. It follows that the maximal order of a finite subgroup of an extended Bianchi group is $12h_{2,m}$.

Assume that $\Gamma$ is reflective and the Coxeter polyhedron $P$ of $\Gamma_r$ has $l$ cusps. Then $\sym$ is finite and the action of any subgroup $F\le\sym$ on the set of cusps of $P$ has to have at most $h_m$ orbits (corresponding to the cusps of  $\Hy^3/\Gamma$). As $F$ is isomorphic to a finite subgroup of $\Gamma$, we conclude that $l \le 12h_m$ or $\le 12h_{2,m}h_m$ for $\Gamma = \Bi$ or $\hBi$, respectively. This proves conditions (1) and (2).

In the quasi-reflective case assume that $v$ is as above and $\Pi_v$ is the corresponding reflection plane. Then each cusp pair $(v,v')$ can be completed to a cusp of the (infinite volume) quotient space $\Hy^3/\Gamma_r$ and all these cusps are different. Moreover, we can identify these cusps with the points on the boundary $\partial\Pi_v$, which is represented by a circle $C$ at infinity where $\sym$ acts by affine transformations.

We denote this set of cusps (or cusp pairs) by $\mathcal{S}$, and let $l$ be its cardinality. The set $\mathcal{S}$ is not stable under the action of $\mathrm{Sym}(P)$, but all we need is to look at the transformations $\gamma\in\sym$ that may identify any of the two cusps from $\mathcal{S}$. These are of three possible types:
\begin{itemize}
\item[(a)] stabilising the circle $C$ and hence in a finite subgroup $F < \sym$;
\item[(b)] mapping $C$ to a tangent circle. (From the solution of the kissing number problem in the plane we know that there are $\le 6$ such $\gamma$'s. The images of $p\in \mathcal{S}$ under these transformations can be also obtained by rotations of $C$ by the angle $\pi/3$, if these rotation are in $\sym$.)
\item[(c)] mapping $C$ to a circle $C'$ which intersects $C$ orthogonally. (An elementary geometrical argument shows that if $p, p' \in \mathcal{S}$ are identified by such $\gamma$ then the subset of the cusps from $\mathcal{S}$ which are equivalent to $p$ consists of $2$ or $4$ elements that can be also identified by a rotation of $C$.)
\end{itemize}
From this we conclude that the maximal size of the orbit of the action of $\sym$ on $\mathcal{S}$ is $12$ in $\Gamma = \Bi$ case and $12h_m$ if $\Gamma = \hBi$. It remains to note that one cusp of $\Hy^3/\Gamma$ always comes from the singular vertex of $P$, which is not adjacent to $v$ by the assumption. Hence we obtain that  $l/12$ or $l/12h_{2,m} \le h_m - 1$, depending on $\Gamma = \Bi$ or $\hBi$, respectively.
\end{proof}

The bounds of Proposition~\ref{QR-cuspcnt} are often not optimal but still quite effective for proving non-reflectiveness (or non-quasi-reflectiveness) of certain groups. In the next section we will apply this proposition and also another known criterion in order to finish the proof of the classification results.

\section{Proofs of the theorems}\label{(Q)reflbianchi}


In Section~\ref{QRbackground} we obtained an upper bound on the value of $m$ for which the groups $Bi(m)$ and $\widehat{Bi}(m)$ may be quasi-reflective. This bound appears to be identical to the bound in \cite{belolipetsky09} for the fields of definition of non-cocompact arithmetic Kleinian reflection groups. Thus, by the result of computation in \cite{belolipetsky09}, we have a list of $882$ fields which can be fields of definition of
reflective/quasi-reflective Bianchi or extended Bianchi groups.

This list of fields can be filtered further according to Proposition~\ref{classfilter}. Using GP/PARI software we obtain that out of the $882$ fields there are:
\begin{itemize}
\item $65$ candidates for reflective Bianchi groups;
\item $188$ candidates for reflective extended Bianchi groups;
\item $203$ candidates for quasi-reflective Bianchi groups; and
\item $204$ candidates for quasi-reflective extended Bianchi groups.
\end{itemize}
The largest value of $m$ in these lists is $7035$. Of the $204$ fields that pass the most permissive restriction on $C(\Am)$, there is only one that does not satisfy Proposition~~\ref{classfilter}(3). It corresponds to $m = 2379$ for which $C(\Am) = (\Z/4\Z)^2$. The full lists can be found in \cite{mcleod_thesis}.

We now apply Vinberg's algorithm to the quadratic forms given in Section~\ref{background} and iteratively generate the mirrors of reflections in an extended Bianchi group using a computer program. The algorithm terminates with a reflective extended Bianchi group for $m \le 21$, $m = 30$, $33$ and $39$. This provides the list of groups in Theorem~\ref{thm2}. In each of the cases we can apply Lemmas~\ref{lem:refllengths} and \ref{lem:sha} to determine when $\hBi$ is reflective, and $\Bi$ is not. This occurs for $m = 14$, $17$, $21$, $30$, $33$, and $39$, which leads to the list of groups in Theorem~\ref{thm1}. If the algorithm does not terminate after a reasonable number of iterations we try to show that the group is not reflective. In order to do so we count the number of cusps produced by the algorithm and apply Proposition~\ref{QR-cuspcnt} or show that the Coxeter polyhedron has an infinite order symmetry associated to a loxodromic isometry of $\Hy^3$. The second criterion was previously used by Bugaenko in \cite{bugaenko92}. We will illustrate it with a quasi-reflective example later on. After applying Proposition~\ref{QR-cuspcnt} we are left with $10$ 
candidates for reflective Bianchi groups and $120$ 
candidates for reflective extended Bianchi groups (excluding the ones which are already considered). Note that these lists depend on the depth of the application of Vinberg's algorithm but regardless of it some groups would remain. In each of the remaining cases we were able to produce the corresponding loxodromic isometries and hence to prove non-reflectiveness. We refer to \cite{mcleod_thesis} for the details of the computation.

\medskip

Most of the results of our computations coincide with the earlier results of Ruzmanov and Shaiheev \cite{ruzmanov90, shaiheev90} but there are also some differences which we are going to discuss now. Most notably, the group $\widehat{Bi}(33)$ (Figure~\ref{fig2}) does not appear in Shaiheev's list. The proof that it is reflective is obtained via Vinberg's criterion (Propositions \ref{bugaenko1} and \ref{bugaenko2}), and is expounded in Table~\ref{proofofm=33}, while the list of vectors normal to the mirrors of the fundamental domain of $\widehat{Bi}(33)$ is given by Table~\ref{vecs33}.

\begin{table}[ht]
\centering
\caption{\label{vecs33}Vectors normal to the mirrors of the fundamental domain of $\widehat{Bi}(33)$.}
\begin{tabular}{l}
$m=33$\\
\hline
(0,0,-1,0)\\
(1,0,1,0)\\
(0,0,0,-1)\\
(33,0,0,1)\\
(-1,1,0,0)\\
(16,2,1,1)\\
(6,6,3,1)\\
(8,4,1,1)\\
(11,3,1,1)\\
(11,11,0,2)\\
(99,33,0,10)\\
(121,22,0,9)\\
(90,18,3,7)\\
(37,8,0,3)\\
(264,66,0,23)
\end{tabular}
\end{table}

To show that this polyhedron has finite volume, consider the edge (elliptic subdiagram) consisting of the mirrors labelled $1$ and $5$. By inspection of the Coxeter diagram in Figure~\ref{m=33} we can see that the vertices $1$ and $5$ are part of an elliptic subdiagram $1, 3, 5$, and another elliptic subdiagram $1, 5, 10$ (both subdiagrams are three copies of the elliptic diagram $A_1$ of rank $1$), and no other elliptic subdiagrams of rank $3$ or parabolic subdiagrams of rank $2$. Moreover, we can see that $1,3,5,10$ do not form an elliptic subdiagram of rank $4$. Therefore the edge of the fundamental polyhedron formed by the intersection of $\Pi_1$ and $\Pi_5$ has, at each end, a vertex that is within the interior of hyperbolic space. All other elliptic subdiagrams of rank $2$ can be handled in the same manner, which shows that the Coxeter diagram has finite volume, and the presence of parabolic subdiagrams confirms that $\widehat{Bi}(33)$ is non-compact. By Lemma~\ref{lem:refllengths} we can identify the mirrors that are not in $Bi(33)$, and these correspond to the filled vertices in the figure. By Lemma~\ref{lem:sha} we conclude that $Bi(33)$ is not reflective.

\begin{table}[ht]
\caption{\label{proofofm=33}Elliptic subdiagrams of rank $2$ of the Coxeter diagram of $\widehat{Bi}(33)$ and their completions to either elliptic subdiagrams of rank $3$ or parabolic subdiagrams of rank $2$. (Only half are listed; the remaining subdiagrams are given by the symmetry of the Coxeter diagram e.g. $1,3$ is equivalent to $1,15$.)}
\centering
\begin{tabular}{l|l|l}
Elliptic diagram & First completion & Second completion\\
\hline
1,3 & 2,4 ; $2 \times \tilde{A}_1$ & 5 ; $ 3 \times A_1$\\
1,4 & 2,3 ; $2 \times \tilde{A}_1$ & 6 ; $ A_1 + B_2 $\\
1,5 & 3 ; $3 \times A_1 $ & 10 ; $ 3 \times A_1$\\
1,8 & 10 ; $ A_1 + B_2 $ & 11 ; $ A_1 + B_2 $\\
1,10& 5 ; $3 \times A_1$ & 8 ; $ A_1 + B_2 $\\
2,3 & 1,4 ; $2 \times \tilde{A}_1$ & 5 ; $ A_1 + A_2 $\\
2,4 & 1,3 ; $2 \times \tilde{A}_1$ & 6 ; $ 3 \times A_1 $\\
2,5 & 3 ; $A_1 + A_2$ & 7 ; $ A_1 + A_2$\\
2,6 & 4 ; $3 \times A_1$ & 9 ; $ B_3 $\\
2,7 & 5 ; $ A_1 + A_2 $ & 8 ; $ A_1 + B_2 $\\
2,8 & 7 ; $A_1 + B_2$ & 9 ; $ B_3 $\\
2,9 & 6 ; $ B_3 $ & 8 ; $ B_3 $\\
3,5 & 1 ; $3 \times A_1$ & 2 ; $A_1 + A_2$\\
4,6 & 1 ; $A_1 + B_2$ & 2 ; $3 \times A_1$\\
5,7 & 2 ; $A_1 + A_2$ & 10 ; $ 3 \times A_1$\\
5,10& 1 ; $3 \times A_1$ & 7 ; $ 3 \times A_1$\\
7,8 & 2 ; $A_1 + B_2$ & 10 ; $ 3 \times A_1$\\
7,10& 5 ; $3 \times A_1$ & 8 ; $ 3 \times A_1$\\
8,10& 1 ; $A_1 + B_2$ & 7 ; $3 \times A_1$\\
\end{tabular}
\end{table}

Let us use this opportunity to correct the mistakes in the data in \cite{shaiheev90} for $m = 17$ and $m = 21$. The vectors which we obtained are given on Table~\ref{fundamentalvectors}, where the differences with \cite{shaiheev90} are indicated. Our computation implies that the group $Bi(21)$ is not reflective in the contrary to what is claimed in \cite{shaiheev90} and some subsequent articles. The computation in all the other cases of Theorems~\ref{thm1} and \ref{thm2} agrees with the previous results.
\begin{table}[ht]
\centering
\caption{Corrections to the list of vectors in $\mathbb{R}^4$ normal to mirrors of the fundamental domains of the reflective extended Bianchi groups $\widehat{Bi}(17)$ and $\widehat{Bi}(21)$. Vectors which are labelled `\dag' are misprinted, while vectors labelled `\ddag' are absent from \cite{shaiheev90}.\label{fundamentalvectors}}
\begin{tabular}{l l}
 $m=17$ & $m=21$ \\
  \hline
 $(0,0,-1,0)$& $(0,0,-1,0)$ \\
 $(1,0,1,0)$&  $(1,0,1,0)$ \\
 $(0,0,0,-1)$&$(0,0,0,-1)$ \\
$(m,0,0,1)$&  $(m,0,0,1)$ \\
$(-1,1,0,0)$&  $(-1,1,0,0)$ \\
 $(8,2,1,1)$ &  $(10,2,1,1)$ \\
 $(4,4,1,1)$&  $(6,3,0,1)$ \\
 $(68,34,17,11)$&   $(6,4,2,1)$ \ddag\\
 $(19,8,0,3)$\dag &   $(42,42,21,8)$ \ddag\\
 $(17,9,1,3)$ &   $(14,14,3,3)$ \ddag\\
 $(136,68,17,23)$&   $(63,63,21,13)$ \ddag\\
 $(85,51,0,16)$&\\
 $(204,102,0,35)$\dag&\\
\end{tabular}
\end{table}

\medskip

We now come to the quasi-reflective case. The Coxeter diagrams of the reflective lattices from Theorems~\ref{thm1} and \ref{thm2} are given in Figures~\ref{fig1} and \ref{fig2}. Using Lemma~\ref{lem:refllengths} we have identified in these figures reflections which are in the extended Bianchi group but are not in the Bianchi group. This information allows us to check which of the reflective extended Bianchi groups contain quasi-reflective Bianchi groups. It occurs precisely when the only reflections represented by filled vertexes of a Coxeter diagram in Figures~\ref{fig1} and \ref{fig2} are those that bound a single cusp. The Bianchi group is quasi-reflective of rank one if the filled vertexes form one cusp pair, and of rank two if all four vertexes are filled. We see that the latter occurs for $m = 14$, $17$, and $39$.

It remains to find the quasi-reflective groups that are not finite index subgroups of reflective extended Bianchi groups. Here we return to the lists of candidates from the beginning of this section and apply to them the corresponding part of Proposition~\ref{QR-cuspcnt}. This leaves us with $175$ 
candidates for quasi-reflective Bianchi groups and $183$ 
candidates for quasi-reflective extended Bianchi groups (excluding those which are already checked). In each of these cases we search for the loxodromic isometries in the symmetry group of the Coxeter polyhedron. Loxodromic isometries do not preserve any horospheres, and as such their presence prevents the lattice from being quasi-reflective. Let us illustrate this procedure for $m = 35$.

Vinberg's algorithm applied to $\widehat{Bi}(35)$ does not terminate, and therefore the Coxeter diagram of the reflection subgroup is infinite. We generated the first $29$ vectors, and the Coxeter diagram of these is presented in Figure~\ref{Bi(35)}.
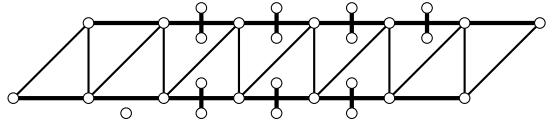
\begin{figure}[h!]
\centering
\begin{tikzpicture}

    \coordinate (one) at (0, 0);
    \coordinate (two) at (1, 0);
    \coordinate (three) at (0.5, 0.2);
    \coordinate (four) at (0.5, -0.2);
    \coordinate (five) at (1, -1);
    \coordinate (six) at (-1, -1);
    \coordinate (seven) at (0, -1);
    \coordinate (eight) at (0.5, -1.2);
    \coordinate (nine) at (-0.5, -1.2);
    \coordinate (ten) at (2, 0);
    \coordinate (eleven) at (-1, 0);
    \coordinate (twelve) at (0.5, -0.8);
    \coordinate (thirteen) at (-0.5, -0.8);
    \coordinate (fourteen) at (1.5, 0.2);
    \coordinate (fifteen) at (-0.5, -0.2);
    \coordinate (sixteen) at (2, -1);    
    \coordinate (seventeen) at (-2, -1);    
    \coordinate (eighteen) at (1.5, -0.2);
    \coordinate (nineteen) at (-0.5, 0.2);
    \coordinate (twenty) at (1.5, -0.8);
    \coordinate (twentyone) at (-1.5, -1.2);
    \coordinate (twentytwo) at (3, 0);
    \coordinate (twentythree) at (-2, 0);
    \coordinate (twentyfour) at (1.5, -1.2);
    \coordinate (twentyfive) at (2.5, -0.2);
    \coordinate (twentysix) at (3, -1);    
    \coordinate (twentyseven) at (-3, -1);    
    \coordinate (twentyeight) at (2.5, 0.2);    
    \coordinate (twentynine) at (4, 0);    

\draw[ultra thick] (one) -- (two);
\draw[thick] (one) -- (six);
\draw[thick] (one) -- (seven);
\draw[ultra thick] (one) -- (eleven);
\draw[thick] (two) -- (five);
\draw[thick] (two) -- (seven);
\draw[ultra thick] (two) -- (ten);
\draw[ultra thick] (three) -- (four);
\draw[ultra thick] (five) -- (seven);
\draw[thick] (five) -- (ten);
\draw[ultra thick] (five) -- (sixteen);
\draw[ultra thick] (six) -- (seven);
\draw[thick] (six) -- (eleven);
\draw[ultra thick] (six) -- (seventeen);
\draw[ultra thick] (eight) -- (twelve);
\draw[ultra thick] (nine) -- (thirteen);
\draw[thick] (ten) -- (sixteen);
\draw[ultra thick] (ten) -- (twentytwo);
\draw[thick] (eleven) -- (seventeen);
\draw[ultra thick] (eleven) -- (twentythree);
\draw[ultra thick] (fourteen) -- (eighteen);
\draw[ultra thick] (fifteen) -- (nineteen);
\draw[thick] (sixteen) -- (twentytwo);
\draw[ultra thick] (sixteen) -- (twentysix);
\draw[thick] (seventeen) -- (twentythree);
\draw[ultra thick] (seventeen) -- (twentyseven);
\draw[ultra thick] (twenty) -- (twentyfour);
\draw[thick] (twentytwo) -- (twentysix);
\draw[ultra thick] (twentytwo) -- (twentynine);
\draw[thick] (twentythree) -- (twentyseven);
\draw[ultra thick] (twentyfive) -- (twentyeight);
\draw[thick] (twentysix) -- (twentynine);

    \filldraw[fill=white] (one) circle (2pt);
    \filldraw[fill=white] (two) circle (2pt);
    \filldraw[fill=white] (three) circle (2pt);
    \filldraw[fill=white] (four) circle (2pt);
    \filldraw[fill=white] (five) circle (2pt);
    \filldraw[fill=white] (six) circle (2pt);
    \filldraw[fill=white] (seven) circle (2pt);
    \filldraw[fill=white] (eight) circle (2pt);
    \filldraw[fill=white] (nine) circle (2pt);
    \filldraw[fill=white] (ten) circle (2pt);
    \filldraw[fill=white] (eleven) circle (2pt);
    \filldraw[fill=white] (twelve) circle (2pt);
    \filldraw[fill=white] (thirteen) circle (2pt);
    \filldraw[fill=white] (fourteen) circle (2pt);
    \filldraw[fill=white] (fifteen) circle (2pt);
    \filldraw[fill=white] (sixteen) circle (2pt);
    \filldraw[fill=white] (seventeen) circle (2pt);
    \filldraw[fill=white] (eighteen) circle (2pt);
    \filldraw[fill=white] (nineteen) circle (2pt);
    \filldraw[fill=white] (twenty) circle (2pt);
    \filldraw[fill=white] (twentyone) circle (2pt);
    \filldraw[fill=white] (twentytwo) circle (2pt);
    \filldraw[fill=white] (twentythree) circle (2pt);
    \filldraw[fill=white] (twentyfour) circle (2pt);
    \filldraw[fill=white] (twentyfive) circle (2pt);
    \filldraw[fill=white] (twentysix) circle (2pt);
    \filldraw[fill=white] (twentyseven) circle (2pt);
    \filldraw[fill=white] (twentyeight) circle (2pt);
    \filldraw[fill=white] (twentynine) circle (2pt);

\end{tikzpicture}
\caption{\label{Bi(35)} Coxeter diagram of the first 29 vectors associated to the Bianchi group $\widehat{Bi}(35)$. (Broken line branches intentionally omitted).}
\end{figure}
This subdiagram has translational symmetry, and we can determine a transformation of $\Hy^3$ that acts on the normal vectors to faces to effect such a symmetry. In the coordinates defined in Section~\ref{background} such a transformation is given by the matrix
$$ g =
\begin{pmatrix}
100 & 429 & -10 & -1230\\
99 & 421 & -9 & -1212\\
30 & 129 & -2 & -369\\
30 & 128 & -3 & -368
\end{pmatrix}.
$$
It is easy to check that this matrix represents a loxodromic transformation of $\Hy^3$ and that it preserves the integral lattice $L_m$. It is also clear that $g$ is not given by a product of reflections because otherwise the required reflections would have been already produced by Vinberg's algorithm. Hence we conclude that $\widehat{Bi}(35)$ is not quasi-reflective.

There remain only two fields that not yet excluded and for which none of the symmetries of the Coxeter diagrams correspond to loxodromic isometries of hyperbolic space. These correspond to $m = 23$ and $m = 31$. Partial Coxeter diagrams of the reflection subgroups of $\widehat{Bi}(23)$ and $\widehat{Bi}(31)$ are presented in Figure~\ref{Bi(23,31)}. In the same way as for $\widehat{Bi}(35)$ we can compute the matrices which effect the translational symmetries of these Coxeter diagrams. All four translations are parabolic isometries of hyperbolic space, and for each diagram the two isometries are linearly independent while preserving a common point on the boundary. Hence these are two new examples of rank $2$ quasi-reflective extended Bianchi groups and the proof of Theorem~\ref{thm3} in now complete. \qed

\bibliographystyle{math}
\bibliography{references}

\end{document}